\documentclass[11pt,leqno]{amsart}

\usepackage{hyperref}

\newtheorem{theorem}{Theorem}[section]
\newtheorem{lemma}[theorem]{Lemma}
\newtheorem{proposition}[theorem]{Proposition}
\newtheorem{corollary}[theorem]{Corollary}

\title{ On words that are concise in residually finite groups}
\author{Cristina Acciarri}

\address{Cristina Acciarri:  Department of Mathematics, University of Brasilia,
Brasilia-DF, 70910-900 Brazil}
\email{acciarricristina@yahoo.it}

\author{Pavel Shumyatsky} 

\address{Pavel Shumyatsky: Department of Mathematics, University of Brasilia,
Brasilia-DF, 70910-900 Brazil}

\email{pavel@unb.br}

\keywords{Residually finite groups, multilinear commutators, concise words}
\subjclass[2010]{Primary  20F10;  Secondary  20E26}

\thanks{This research was supported by CNPq-Brazil.}

\begin{document}
\begin{abstract} A group-word $w$ is called concise if whenever the set of $w$-values in a group $G$ is finite it always follows that the verbal subgroup $w(G)$ is finite. More generally, a word $w$ is said to be concise in a class of groups $X$ if whenever the set of $w$-values  is finite for a group $G\in X$, it always follows that $w(G)$ is finite. P. Hall asked whether every word is concise. Due to Ivanov the answer to this problem is known to be negative. Dan Segal asked whether every word is concise in the class of residually finite groups. In this direction we prove that if $w$ is a multilinear commutator and $q$ is a prime-power, then the word $w^q$ is indeed concise in the class of residually finite groups. Further, we show that in the case where $w=\gamma_{k}$ the word $w^{q}$ is boundedly concise in the class of residually finite groups. It remains unknown whether the word $w^q$ is actually concise in the class of all groups.

\end{abstract}

\maketitle

\section{Introduction}
Let $w$ be a group-word in $n$ variables, and let $G$ be a group. The verbal subgroup $w(G)$ of $G$ determined by $w$ is the subgroup generated by the set $G_w$ consisting of all values $w(g_1,\ldots,g_n)$, where $g_1,\ldots,g_n$ are elements of $G$.  A word $w$ is said to be concise if whenever $G_w$ is finite for a group $G$, it always follows that $w(G)$ is finite. More generally, a word $w$ is said to be concise in a class of groups $X$ if whenever $G_w$ is finite for a group $G\in X$, it always follows that $w(G)$ is finite. P. Hall asked whether every word is concise, but  later Ivanov proved that this problem has a negative solution in its general form \cite{ivanov} (see also \cite[p.\ 439]{ols}). On the other hand, many relevant words are known to be concise. For instance, it was shown in \cite{jwilson} that the multilinear commutator words are concise. Such words are also known under the name of outer commutator words and are precisely the words that can be written in the form of multilinear Lie monomials. Merzlyakov showed that every word is concise in the class of linear groups \cite{merzlyakov} while Turner-Smith proved that every word is concise in the class of  residually finite groups all of whose quotients are again residually finite \cite{TS2}. There is an open problem, due to Dan Segal \cite[p.\ 15]{Se}, whether every word is concise in the class of residually finite groups.

In the present paper we obtain the following result.

\begin{theorem}
\label{T1}
Let $w$ be a multilinear commutator word and  $q$ a prime-power.
The word $w^q$ is concise in the class of residually finite groups.
\end{theorem}

It remains unknown whether the word $w^q$ is actually concise in the class of all groups. Following \cite{brakrashu} we say that a word $w$ is boundedly concise in a class of groups $X$ if for every integer $m$ there exists a number $\nu=\nu(X,w,m)$ such that whenever $|G_w|\leq m$ for a group $G\in X$ it always follows that $|w(G)|\leq\nu$. Fern\'andez-Alcober and Morigi \cite{fernandez-morigi} showed that every word which is concise in the class of all groups is actually boundedly concise. Moreover they showed that whenever $w$ is a multilinear commutator word having at most $m$ values in a group $G$, one has $|w(G)|\leq (m-1)^{(m-1)}$. In view of our Theorem \ref{T1} it would be interesting to determine whether the word $w^q$ is boundedly concise in the class of residually finite groups. In this direction we obtain the following theorem.

\begin{theorem} \label{T2}
Let   $w=\gamma_{k}$ be the $k$th lower central word and $q$ a prime-power. The word $w^{q}$ is boundedly concise in the class of residually finite groups.
\end{theorem}

Recall that the word $\gamma_{k}$ is defined inductively by the formulae
\[
\gamma_1=x_1,
\qquad
\gamma_k=[\gamma_{k-1},x_k]=[x_1,\ldots,x_k],
\quad
\text{for $k\ge 2$.}
\]
The corresponding verbal subgroup $\gamma_k(G)$ is the familiar $k$th term of the lower central series of $G$. It remains unknown whether the word $\gamma_{k}^q$ is concise in the class of all groups. 

The proofs of both Theorem \ref{T1} and Theorem \ref{T2} are based on techniques created by Zelmanov in his solution of the Restricted Burnside Problem (see \cite{korea} or \cite{orko}). Some adjustments to those techniques are described in \cite{eueu}. We have been unable to decide whether the condition that $q$ is a prime-power is necessary in our results.

Throughout the paper we use the abbreviation, say, ``$(a,b,\dots)$-bounded'' for ``bounded above in terms of $a,b,\dots$ only''.

\section{Preliminaries}
In this section we collect some results that will be useful later.  Recall that a group $G$ is said to be perfect if $G=G'$. As usual if $L,M\leq G$, we write $[L,M]$ to denote the subgroup generated by all commutators $[l,m]$ where $l\in L$ and $m\in M$.
\begin{lemma}
\label{lemma1}
Let $G$ be a perfect group and $A$  a normal abelian subgroup of $G$. Then $[A,G]=[A,G,G]$.
\end{lemma}
\begin{proof}

Without loss of generality  assume that $[A,G,G]=1$.  Note that $[G,A,G]=[A,G,G]=1$. In view of the Three  Subgroups Lemma \cite[5.1.10]{rob}, we have $[G,G,A]=1$. Since $G$ is perfect, we conclude that $[A,G]=1$, as desired.
\end{proof}

 Let $w$ be a group-word and $G$ a group. A subgroup $N\leq G$ will be called $w$-subgroup  if $N$ is generated by cyclic subgroups contained in $G_w$. By weight of a multilinear commutator we mean the number of variables involved in the word. It is clear that any multilinear commutator  $w$ of weight $k\geq 2$ can be written in the form $w=[w_{1},w_{2}]$ where $w_{1}$ and $w_{2}$ are multilinear commutators of smaller weights.
 
\begin{lemma}
\label{lemma2}
Let $G$ be a perfect group  and $A$  a normal  abelian  subgroup of $G$. Then $[A,G]$ is a $w$-subgroup for any multilinear commutator $w$.
\end{lemma}

\begin{proof} Choose a multilinear commutator word $w$ and let $k$ be the weight of $w$. If $k=1$, then $w=x$ and $[A,G]$ is a $w$-subgroup. Thus, assume that $k\geq 2$ and show that $[A,G]$ is a $w$-subgroup using induction on $k$. Let $w=[w_{1},w_{2}]$, where $w_{1}$ and $w_{2}$ are two multilinear commutators of smaller weights, say $k_{1}$ and $k_{2}$, such that $k=k_{1}+k_{2}$. By the  inductive hypothesis  $[A,G]$ is a $w_{1}$-subgroup in $G$. Moreover since $w$ is a multilinear commutator and $G$ is perfect, we have $G = w_{2}(G)$, and so $G$ is generated by $w_{2}$-values. By Lemma \ref{lemma1} we know that  $[A,G]=[A,G,G]$. Hence $[A,G]$ is generated by elements of the form $[x,y]$, where $\langle x\rangle\subseteq G_{w_{1}}$ and $y\in G_{w_{2}}$. Since $[A,G]$ is abelian, we obtain that 
$$[x,y]^{j}=[x^{j},y]\in [G_{w_{1}},G_{w_{2}}]\subseteq G_{w}$$ for any $j$. 
Thus we conclude that $[A,G]$ is a $w$-subgroup, as desired.  
\end{proof}

Another result of similar nature is the following lemma.

\begin{lemma}
\label{lemma3}
Let $K$ be a normal subgroup of a group $G$ and $w=[w_{1},w_{2}]$ a multilinear commutator word. Suppose that $K$ is nilpotent of class two.  If  $K/Z(K)$ is  a $w_{1}$-subgroup in $G/Z(K)$ and if it is generated by $w_{2}$-values in $G/Z(K)$,   then $K'$ is a $w$-subgroup in $G$.
\end{lemma}
\begin{proof}
By the hypothesis $K$ has a generating set $X$ such that $x^{i}\in G_{w_{1}}Z(K)$, for every $x\in X$ and every integer $i$. Similarly $K$ has a generating set $Y$ such that $y\in G_{w_{2}}Z(K)$ for every $y \in Y$.

Since $K$ is nilpotent of class two, it follows that $K'$ is  generated by elements of the form $[x,y]$, where $x\in X$ and $y \in Y$. Now we have  $$[x,y]^{i}=[x^{i},y]\in [G_{w_{1}}Z(K),G_{w_{2}}Z(K)]\subseteq G_{w}$$ for any $i$.
Thus $K'$ is a $w$-subgroup in $G$, as desired.   
\end{proof}

The next lemma ensures us that in any soluble-by-finite group $G$ we can find a term of the derived series of $G$ which is a perfect group.  
\begin{lemma}\label{lemma4} Let $G$ be a group having a normal soluble subgroup with finite index $n$  and derived length $d$. Then $G^{(i)}=G^{(i+1)}$ for some integer $i\leq d+n$.
\end{lemma}
\begin{proof}
It is clear that  any soluble quotient of $G$ must have derived length at most $d+n$. Therefore  we have $G^{(i)}=G^{(i+1)}$ for some $i\leq d+n$.   
\end{proof}

To any finite ordered set of subgroups of a group $G$, say $\{N_{1},\ldots, N_{s}\}$, we associate the \emph{triple} $(s,s-i,s-j)$, where $s$ is the number  of subgroups in the set, $i$ is the maximal index such that $N_{i} $ is central in $G$ and $j$ is the maximal index such that $N_{j}$ is abelian.  If there are no central  subgroups, or no abelian subgroups in the set, then we put $i=0$ or $j=0$, respectively. The set of all such possible triples is naturally endowed with the lexicographical order. In the sequel we say that one triple is smaller than another meaning that it is smaller with respect to the lexicographical order. Our goal in the remaining part of this section is to prove that the verbal subgroup $w(G)$ of a soluble-by-finite group $G$ possesses  a normal series with some very specific properties. Similar series were considered in \cite{brakrashu} and \cite{fernandez-morigi}. We start with the case where $G$ is perfect.

\begin{proposition}
\label{perfect case}
There exist $(d,n)$-bounded integers $m$ and $k$ with the following property: 
Let $G$ be  a perfect group having a normal soluble subgroup  of finite index $n$  and with derived length $d$. Then $G$ has a characteristic series $$1=T_{1}\leq T_{2}\leq \cdots \leq T_{m}\leq G$$ such that every quotient $T_{i}/T_{i-1}$ is an abelian $w$-subgroup in $G/T_{i-1}$ for every multilinear commutator word $w$ and  the index $[G:T_{m}]$ is at most  $k$. 
\end{proposition}
\begin{proof}
Since $G=G'$, we have $G=w(G)$ for any $w$.  Let $S$ be the soluble radical of $G$ and let $d_{1}$ be  the derived length of $S$. It is clear that   $d_{1}\leq d+n$.  If $d_{1}=0$, then $G$ is finite of order $n$ and  so there is nothing to prove. 

Thus, suppose that $d_{1}\geq 1$ and use induction. Let $A$ be the last nontrivial term of the derived series of $S$.  
We consider the quotient group $\overline{G}=G/A$ and, as usual, if $X\leq G$, we  denote by  $\overline{X}$  the image of $X$ in $\overline{G}$. By the inductive hypothesis $\overline{G}$ has a  characteristic series $$\overline1=\overline{K_{1}}\leq\overline{K_{2}} \leq \cdots \leq \overline{K_{s}}\quad \quad (*)$$ such that every quotient $\overline{K_{i}}/\overline{K}_{i-1}$ is an abelian $w$-subgroup in $\overline{G}/\overline{K}_{i-1}$  for every $w$ while $s$ and the index  $[\overline{G}: \overline{K_{s}}]$ are  both $(d,n)$-bounded.  We choose the series $(*)$ in such a way that the triple  $(s,s-i,s-j)$ associated with  the set $\{K_{1}=A,K_{2},\ldots,K_{s}\}$ is  as small as possible.  We now argue by induction on the triple.   

If $K_{s}$ is central in $G$, then the index $[G:Z(G)]$ is $(d,n)$-bounded. Thus, by Schur's theorem \cite[10.1.4]{rob} $|G|$ is finite and $(d,n)$-bounded so there is nothing to prove.
Suppose that $K_{s}$ is not central in $G$. By the definition of the triple,  $K_{i+1}$ is the first term which is not central in $G$. We define a subgroup $D$ in the following way: 
\begin{itemize}
\item[(1)] if $K_{i+1}$ is abelian, then  $D=[K_{i+1},G]$;
\item[(2)] if $K_{i+1}$ is not abelian, then  $D=[K_{i+1},K_{i+1}]$.
\end{itemize}
 In the case (1) Lemma \ref{lemma2} tells us that  $D$ is a $w$-subgroup in $G$ for every multilinear commutator word $w$. In  the case (2) the subgroup  $K_{i+1}$ is nilpotent of class two and we are in the position to apply Lemma \ref{lemma3}. Indeed, note that $K_{i+1}/Z(K_{i+1})$ is isomorphic to a quotient  group of  $K_{i+1}/K_{i}$, which is a $w$-subgroup in $G/K_{i}$ for every multilinear commutator word $w$ and moreover $K_{i}\leq Z(K_{i+1})$. Thus, also in  case (2) we conclude that $D$ is  a $w$-subgroup in $G$ for all $w$. It is clear that $D$ is  abelian.

Next, we consider the quotient group $\widetilde{G}=G/D$. Observe that in the case (1) the subgroup  $\widetilde{K_{i+1}}$ is central in $\widetilde{G}$ and, in the case (2),  the subgroup $\widetilde{K_{i+1}}$ is abelian. Therefore in both cases the triple associated to the  set $\{\widetilde{K_{1}},\widetilde{K_{2}},\ldots, \widetilde{K_{s}}\}$ is smaller.  By induction on the triple we deduce that $\widetilde{G}$ has a characteristic series 
$$\widetilde{1}=\widetilde{D_{1}}\leq\widetilde{D_{2}}\leq\cdots\leq\widetilde{D_{l}}$$ such that  every quotient $\widetilde{D_{i}}/\widetilde{D_{i-1}}$ is an abelian $w$-subgroup in $\widetilde{G}/\widetilde{D_{i-1}}$ for all $w$, and the index $[\widetilde{G}: \widetilde{D_{l}}]$ and $l$ are $(d,n)$-bounded. Put $T_{1}=1, T_{2}=D, \ldots, T_{l+1}=D_{l}$. The characteristic series 
$$1=T_{1}\leq T_{2}\leq\cdots\leq T_{l+1}$$ has all the required properties. This completes the proof.
\end{proof}

Now we analyze  the general case where $G$ is a soluble-by-finite group. 

\begin{proposition}
\label{general case}
Let $w$ be a multilinear commutator.  There exist a $(d,n,w)$-bounded integer $s$ and a $(d,n)$-bounded integer $h$ with the following property: 
Let $G$ be  a group having a normal soluble subgroup  of finite index $n$  and derived length $d$. Then $G$ has a  series of normal subgroups  
$$1=T_{1}\leq T_{2}\leq \cdots \leq T_{s}=w(G)$$ such that every quotient $T_{i}/T_{i-1}$ is an  abelian $w$-subgroup  in $G/T_{i-1}$, except possibly one quotient whose order is at most $h$. 
\end{proposition}

\begin{proof}
By Lemma \ref{lemma4} there exists an integer $j\leq d+n$ such that $G^{(j)}=G^{(j+1)}$. Since $G^{(j)}$ is perfect,  we have $G^{(j)}=w(G^{(j)})$, and  so $G^{(j)} \leq w(G)$.   By Proposition \ref{perfect case} $G^{(j)}$ has  characteristic subgroups $$1=S_{1}\leq S_{2}\leq \cdots \leq S_{m} $$ such that each quotient $S_{i}/S_{i-1}$ is an abelian $w$-subgroup in $G^{(j)}/S_{i-1}$ while   the index $[G^{(j)}:S_{m}]$ and $m$ are $(d,n)$-bounded. 

The  quotient $\overline{G}=G/G^{(j)}$ is  soluble  with derived length at most $j$. It follows  from \cite[Theorem B]{fernandez-morigi}  that $\overline{G}$ has a normal series 
$$\overline{1}=\overline{P_{1}}\leq\overline{P_{2}} \leq \cdots\leq\overline{P_{t}}=w(\overline{G})$$ such that   every quotient $\overline{P_{i}}/\overline{P}_{i-1}$  is an abelian  $w$-subgroup  in $\overline{G}/\overline{P}_{i-1}$ and $t$ is $(d,n,w)$-bounded. Put $T_{1}=S_{1},\ldots,T_{m}=S_{m}, T_{m+1}=G^{(j)}, T_{m+2}=P_{2}, \ldots, T_{m+t}=P_{t}$. By the construction the index $[T_{m+1}:T_{m}]$ is $(d,n)$-bounded. Moreover every quotient  $T_{i}/T_{i-1}$  with $i\neq m+1$ is an abelian $w$-subgroup  in $G/T_{i-1}$, and $m+t$ is $(d,n,w)$-bounded. Therefore  $1=T_{1}\leq T_{2}\leq\cdots \leq T_{m+t}=w(G)$ is the  series with the desired properties.  This completes the proof.  
\end{proof}

We conclude this section with the following  corollary  of Proposition \ref{general case}.

\begin{corollary}
\label{finitexpo}
Let $G$ be a soluble-by-finite group, $w$ a multilinear commutator word, and $m\geq 1$ an integer. Suppose that $G$ has only  finitely many $w^m$-values. Then $w(G)$ has finite exponent. If $G$ has only $t$ values of $w^{m}$ and  the soluble radical of $G$ has index $n$ and derived length $d$, then  the exponent of $w(G)$ is $(d,m,n,t,w)$-bounded.
\end{corollary}

\begin{proof}
It follows from Theorem \ref{general case}  that $G$ has a  series, 
$$ 1=T_1\leq T_2\leq\cdots \leq T_s=w(G)$$
such that  $s$ is a $(d,n,w)$-bounded number and every quotient $T_i/T_{i-1}$ is an abelian $w$-subgroup  in $G/T_{i-1}$, except possibly one quotient whose order is $(d,n)$-bounded. By the hypothesis $G$ has only $t$  values of $w^m$.  Therefore every quotient $T_i/T_{i-1}$ must have finite exponent bounded  in terms of $m$ and $t$, except  possibly one finite quotient whose order is $(d,n)$-bounded. Thus we  conclude that $w(G)$ has finite exponent bounded in terms of $d,m,n,t$ and $w$, as claimed. 
\end{proof}


\section{Proof of the main results}
The aim of this section is to prove  Theorem \ref{T1} and Theorem \ref{T2}. For the reader's convenience we recall some useful facts.

\begin{lemma}
\label{lemma5}
Let $w$ be  any word and $G$ a group such that  $G_w$ is finite.  Then  $w(G)$ is finite if and only if  it is periodic. Moreover if $G$ has precisely $m$ $w$-values and  $w(G)$ has exponent  $e$, then  the order of $w(G)$ is $(e,m)$-bounded. 
\end{lemma}
\begin{proof}
Assume that $G_w$ is a finite set of $m$ elements. Let $C=C_G(G_w)$. Since the set $G_w$ is normal, it follows that the index of $C$ in $G$ is at most $m!$. Since $w(G)$ is generated by at most $m$ $w$-values,  $C\cap w(G)$ is generated by an $m$-bounded number of elements. Moreover  $C\cap w(G)$ is an abelian subgroup with $m$-bounded index in $w(G)$. Therefore $w(G)$ is finite if and only if $C\cap w(G)$ is periodic. It is clear that  if $C\cap w(G)$ has exponent $e$, then $w(G)$ has $(e,m)$-bounded order. 
\end{proof}

 We need   the following result that can be found in \cite{Shu}. The proof  is based on Lie methods in the spirit of Zelmanov's solution of the Restricted Burnside Problem. 
\begin{theorem}
\label{teoshum}
Let $q$ be  a prime-power and $w$ a multilinear commutator word. Assume that $G$ is a residually finite group  such that any $w$-value in $G$ has order dividing $q$. Then the verbal subgroup $w(G)$ is locally finite.
\end{theorem}

Now we are ready to prove Theorem \ref{T1}.
\begin{proof}[Proof of Theorem \ref{T1}]
Let $G$ be  a residually finite group in which the word $v=w^q$ has only finitely many values. In view of Lemma \ref{lemma5} it is sufficient to show that $v(G)$ is periodic. Choose a normal subgroup $K$ in $G$ such that the index $[G:K]$ is finite and $v(K)=1$. All $w$-values in $K$ have  order dividing $q$. By Theorem \ref{teoshum}  $w(K)$ is locally finite and so in particular it is periodic. We pass to the quotient group $G/w(K)$ and assume that $w(K)=1$. Then $K$ is soluble and so $G$ is soluble-by-finite.    Corollary \ref{finitexpo} now tells us  that  $w(G)$ has finite exponent. Since every $v$-value in $G$ is an element of $w(G)$, it follows that  $v(G)$ is periodic too. The proof is complete.
\end{proof}

Next  we  deal with the proof of Theorem \ref{T2}. The following result will be useful.

\begin{proposition}
\label{bounded}
Let $m,k\geq 1$ and $q$ be a $p$-power for some prime $p$. Assume that $G$ is a finite $m$-generator group such that $[x_{1},\ldots,x_{k}]^{q}=1$ for all $x_1,\dots,x_k\in G$. Then the exponent of $\gamma_{k}(G)$ is $(k,m,q)$-bounded.
\end{proposition}

\begin{proof}
 In the case where $k=2$ this is Lemma 3.3 of \cite{upi}. The general case is actually rather similar to the case $k=2$. By \cite[Lemma 5]{novar} $\gamma_k(G)$ is a $p$-group. Therefore $G$ is metanilpotent. It follows from \cite[Theorem 3.2]{austr} that there exists a number $s$, depending only on $m$, such that every element of $\gamma_k(G)$ is a product of at most $s$ $\gamma_k$-values. Of course this also follows from Segal's theorem \cite{Se2}. Let $x$ be an arbitrary element in $\gamma_k(G)$ and write $x=x_1\dots x_s$, where $x_i$ are $\gamma_k$-values. Let $H=\langle x_1,\dots,x_s\rangle$. By Zelmanov's theory the order of $H$ is $(k,q,s)$-bounded. In particular the order of the element $x$ is $(k,q,s)$-bounded. Since $x$ has been chosen in $\gamma_k(G)$ arbitrarily and since $s$ depends only on $m$, the result follows.
\end{proof}

\begin{proposition}
\label{finite case}
Let $k,m\geq 1$ and $q$ be a prime-power. Let $v=[x_{1},\ldots,x_{k}]^{q}$ and suppose that $G$ is a finite group having precisely $m$ $v$-values. Then  the order of $v(G)$ is $(k,m,q)$-bounded. 
\end{proposition}
\begin{proof}
Without loss of generality we can assume that $G$ is generated by at most $mk$ elements. Indeed, let $\{u_{1},\ldots,u_{m}\}$ be the set of all $v$-values in $G$. Write $u_{i}=[g_{i1},\ldots,g_{ik}]^{q}$ for $i=1,\ldots,m$. Let $K$ be the subgroup of $G$ generated by $g_{11},g_{12},\ldots,g_{mk}$ and note that $v(G)=v(K)$. Hence, without loss of generality we can assume that $G$ is generated by $g_{11},g_{12},\ldots,g_{mk}$.

Let $H$ be the centralizer in $G$ of all $v$-values. Then $[G:H]\leq m!$. Since $G$ is  generated by at most  $mk$ elements and $H$ has $m$-bounded index in $G$, it follows that also $H$ has a $(k,m)$-bounded number of generators. Consider the quotient $\overline{H}=H/Z(H)$. Since $H$ centralizes all  $v$-values, it follows that the law $v\equiv 1$ holds in  $\overline{H}$. Therefore we are in the position to use Proposition \ref{bounded} and deduce that the exponent of $\gamma_{k}(\overline{H})$ is $(k,m,q)$-bounded. Thus $\gamma_{k}(H)/Z(\gamma_{k}(H))$ has $(k,m,q)$-bounded exponent. Mann's  result \cite{mann} tells us that  $[\gamma_{k}(H),\gamma_{k}(H)]$  has $(k,m,q)$-bounded exponent. We pass  to the quotient $G/[\gamma_{k}(H),\gamma_{k}(H)]$ and  assume that $\gamma_{k}(H)$ is abelian. It follows that $H$ is soluble with derived length at most $k+1$. By  Corollary \ref{finitexpo} we conclude that $v(G)$ has  $(k,m,q)$-bounded exponent. Lemma \ref{lemma5} now tells us  that $v(G)$ has $(k,m,q)$-bounded order, as desired.  
\end{proof}

The proof of Theorem \ref{T2} is now immediate.
\begin{proof}[Proof of Theorem \ref{T2}] Let $v=w^{q}$ and choose a residually finite group $G$ in which the word $v$ has at most $m$ values. According to Proposition \ref{finite case} there exists a $(k,m,q)$-bounded number $\nu$ such that the order of $v(K)$ is at most $\nu$ for every finite quotient $K$ of $G$. Hence, every subgroup of finite index of $G$ intersects $v(G)$ by a subgroup of index at most $\nu$ in $v(G)$. Since $G$ is residually finite, we conclude that the order of $v(G)$ is at most $\nu$. The proof is complete.  
\end{proof}



\begin{thebibliography}{99}
\bibitem{brakrashu} S. Brazil, A. Krasilnikov, P. Shumyatsky, Groups with bounded verbal conjugacy classes, \textit{J. Group Theory} {\bf 9} (2006), 127--137.
\bibitem{fernandez-morigi} 
G.\,A. Fern\'andez-Alcober, M. Morigi, Outer commutator words are uniformly concise. \textit{J. London Math. Soc.} \textbf{82} (2010), 581--595.
\bibitem{ivanov} S. V. Ivanov, P. Hall's conjecture on the finiteness of verbal subgroups, \textit{Izv. Vyssh. Ucheb. Zaved.} {\bf 325} (1989), 60--70.
\bibitem{mann}A. Mann, The exponents of central factor and commutator groups, {\it J. Group Theory} {\bf 10} (2007),  435--436.
\bibitem{merzlyakov}Ju.\,I. Merzlyakov, Verbal and marginal subgroups of linear groups, \textit{Dokl. Akad. Nauk SSSR} {\bf 177} (1967), 1008--1011.
\bibitem{ols} A.\,Yu. Ol'shanskii, \textit{Geometry of Defining Relations in Groups}, Mathematics and its applications  {\bf70} (Soviet Series), Kluwer Academic Publishers, Dordrecht, 1991.
\bibitem{rob}D.\,J.\,S. Robinson, \textit{A Course in the Theory of Groups}, Springer, 1996, 2nd edn.
\bibitem{Se} D. Segal, \textit{Words: notes on verbal width in groups}, LMS Lecture Notes \textbf{361}, Cambridge Univ. Press, Cambridge, 2009.
\bibitem{Se2} D. Segal, Closed subgroups of profinite groups, {\it Proc.  London Math.
Soc.} (3) {\bf 81} (2000), 29--54.
\bibitem{upi} P. Shumyatsky, Groups with commutators of bounded order, {\it Proc. Amer. Math. Soc.} {\bf 127}
(1999), 2583--2586.
\bibitem{Shu} P. Shumyatsky, Verbal subgroups in residually finite groups, {\it Quart. J. Math.} {\bf 51} (2000), 523--528.
\bibitem{eueu} P. Shumyatsky,
Applications of Lie ring methods to group theory, in {\it Nonassociative algebra and its applications}, eds. R. Costa, A. Grishkov, H. Guzzo Jr.
and L.\,A. Peresi,
Lecture Notes in Pure and Appl. Math., Vol. 211 (Dekker, New York, 2000),  373--395.

\bibitem{novar}  P. Shumyatsky, A (locally nilpotent)-by-nilpotent variety of groups,  \textit{Math. Proc.  Camb. Phil. Soc.}, {\bf132} (2002), 193-196. 
\bibitem{austr} P. Shumyatsky, On profinite groups in which commutators are Engel, {\it J. Aust. Math. Soc.} {\bf 70} (2001), 1-9. 
\bibitem{TS2} R.\,F. Turner-Smith, Finiteness conditions for verbal subgroups, \textit{J. London Math. Soc.} \textbf{41} (1966), 166--176.
\bibitem{jwilson} J. Wilson, On outer-commutator words, \textit{Can. J. Math.} {\bf 26} (1974), 608--620.
\bibitem{korea} E. Zelmanov, {\it Nil Rings and Periodic Groups}, The Korean Math. Soc., Lecture Notes in Math., Seoul, 1992.
\bibitem{orko} E. Zelmanov, Lie ring methods in the  theory of nilpotent groups,  in {\it Proc. Groups'93/St. Andrews}, (eds: C. M. Campbell et al.), LMS Lecture Note  {\bf 212},  Cambridge Univ. Press, 1995, 567--585. 
\end{thebibliography}
\end{document}